\documentclass[12pt,reqno]{amsart}
\usepackage[left=3cm,top=2cm,right=3cm,bottom=2cm]{geometry}
\usepackage{amssymb}
\usepackage{amsbsy}
\usepackage{epsfig}
\usepackage{tikz}
\usepackage{verbatim}
\usepackage[normalem]{ulem}
\usepackage[pdftex]{hyperref}
\hypersetup{colorlinks=true,linkcolor=blue,citecolor=blue,breaklinks = true}

\tikzstyle{vertex} = [fill,shape=circle,node distance=80pt]
\tikzstyle{edge} = [fill,opacity=.5,fill opacity=.5,line cap=round, line join=round, line width=40pt]
\tikzstyle{elabel} =  [fill,shape=circle,node distance=30pt]

\pgfdeclarelayer{background}
\pgfsetlayers{background,main}

\begin{document}
\title{The Signless Laplacian Matrix of Hypergraphs}
 	
\author[K. Cardoso]{Kau\^e Cardoso} \address{Instituto Federal do Rio Grande do Sul - Campus Feliz, CEP 95770-000, Feliz, RS, Brazil}
\email{\tt kaue.cardoso@feliz.ifrs.edu.br}
 	
\author[V.Trevisan]{Vilmar Trevisan} \address{Instituto de Matem\'atica e Estat\'{\i}stica, UFRGS,  CEP 91509--900, Porto Alegre, RS, Brazil}
\email{\tt trevisan@mat.ufrgs.br} 		


\pdfpagewidth 8.5 in \pdfpageheight 11 in

\newcommand{\h}{\mathcal{H}}
\newcommand{\g}{\mathcal{G}}
\newcommand{\A}{\mathbf{A}}
\newcommand{\B}{\mathbf{B}}
\newcommand{\C}{\mathbf{C}}
\newcommand{\D}{\mathbf{D}}
\newcommand{\M}{\mathbf{M}}
\newcommand{\N}{\mathbf{N}}
\newcommand{\lin}{\mathcal{L}}
\newcommand{\cli}{\mathcal{C}}
\newcommand{\Q}{\mathbf{Q}}
\newcommand{\x}{\mathbf{x}}
\newcommand{\y}{\mathbf{y}}
\newcommand{\z}{\mathbf{z}}
\newcommand{\Ah}{\mathbf{A}(\mathcal{H})}

\theoremstyle{plain}
\newtheorem{Teo}{Theorem}
\newtheorem{Lem}[Teo]{Lemma}
\newtheorem{Pro}[Teo]{Proposition}
\newtheorem{Cor}[Teo]{Corollary}

\theoremstyle{definition}
\newtheorem{Def}{Definition}[section]
\newtheorem{Afi}[Def]{Claim}
\newtheorem{Que}[Def]{Question}
\newtheorem{Exe}[Def]{Example}
\newtheorem{Obs}[Def]{Remark}

\newcommand{\keyword}[1]{\textsf{#1}}

\begin{abstract}
In this paper we define signless Laplacian matrix of a hypergraph and obtain structural properties from its eigenvalues. We generalize several known results for graphs, relating the spectrum of this matrix with structural
parameters of the hypergraph such as the maximum degree, diameter and the chromatic number. In addition, we characterize the complete signless Laplacian spectrum for the class of the power hypergraphs from the spectrum of its base hypergraph.\newline


\end{abstract}

\maketitle

\section{Introduction}

The goal of spectral graph theory is to study structural properties of graphs
by means of eigenvalues and eigenvectors of matrices associated with it.
Researchers, motivated by the success of this theory, have studied many
hypergraph matrices aiming to develop a spectral hypergraph theory. See for
example \cite{Banerjee, Feng, Reff20192, Reff2014, Reff2016, Reff2012,
Rodriguez1, Rodriguez2, Rodriguez3}. In 2012 Cooper and Dutle presented a new
approach and, in their paper \cite{Cooper}, the authors proposed the study of
hypergraphs through tensors, causing a revolution in this area and,
consequently, the study of hypergraph from its matrices has been put aside.
Because determining the spectrum of a tensor has a high computational (as
well as theoretical) cost, the application of this theory has its toll.
Therefore, we believe that the study of hypergraphs through matrices remains
important.

Let $\h$ be a hypergraph whose incidence matrix is $\mathbf{B}(\h)$. The
\emph{signless Laplacian matrix} of $\h$ is defined as $\mathbf{Q}(\h) =
\mathbf{B}\mathbf{B}^T$. The main goal of this paper is the study of this
matrix. We say that the eigenvalues of $\Q$ are  the signless Laplacian
eigenvalues of $\h$. The matrix $\Q$ has many interesting properties such as
being symmetric, non-negative, semi-definite positive and irreducible. Thus,
important theorems such as Perron-Frobenius and Rayleigh Principle can be
inherited directly from matrix theory. In this paper, we prove generalizations
of some results that this matrix has in the context of graphs and,
consequently, it is possible to determine structural properties of the
hypergraph from $\Q$. For example, we show that the number of edges of the
hypergraph can be determined from the sum of its signless Laplacian
eigenvalues. We also show that the number of distinct eigenvalues of $\Q$ is
larger than the diameter of the hypergraph. The spectral radius is bounded by
the degrees of the hypergraph and the chromatic number is bounded from  the
spectral radius. We also show how to determine whether a hypergraph is regular
by analyzing its spectral radius, or its principal eigenvector.

One of the most important properties of the signless Laplacian matrix in the
context of spectral graph theory is the relation between the eigenvalue zero
and the existence of bipartite components in the graph. See Proposition 2.1
of \cite{Towards0}. In an attempt to obtain similar results, we study the
signless Laplacian eigenvalue zero of a hypergraph. We establish the
following result.
\begin{Teo}\label{teo:caracZero} Let $\h = (V,E)$ be a
$k$-graph. If $\lambda=0$ is an eigenvalue of $\Q(\h)$, then $\h$ is
partially bipartite.
\end{Teo}
For the definition of a partially bipartite hypergraph see Section
\ref{sec:strutural-properties}. The converse of Theorem \ref{teo:caracZero}
is not true. For example $\h = (\{1,2,3,4\},\;\{123,124,134,234\})$, has a
partially bipartition $V_1=\{1,2\}$, $V_2=\{3,4\}$ and $V_0=\emptyset$, but
the eigenvalues of $\Q$ are $\rho=9$ and $\lambda = 1$ with multiplicity $3$.
In view of this, we leave here the following question.
\begin{Que}
How to characterize uniform hypergraphs with signless Laplacian eigenvalue zero?	
\end{Que}

As an application of our developed theory, we also study the spectrum of the
signless Laplacian matrix of the class of hypergraphs called power
hypergraphs (see definition in Section \ref{sec:power}). We show how to
construct the whole spectrum of the power hypergraph from the signless
Laplacian eigenvalues of its base hypergraph.

The remaining of the paper is organized as follows. In Section \ref{sec:pre}
we present some basic definitions about hypergraphs and matrices. In Section
\ref{sec:incidence} we study the incidence matrix and exploit some properties
of line and clique multigraphs. In Section \ref{sec:Laplacian} we study the
signless Laplacian matrix, extending many classical results of this matrix to
the context of hypergraphs. In Section \ref{sec:strutural-properties} we
study structural characteristics of a hypergraph such as being regular or
partially bipartite, analyzing its signless Laplacian eigenvalues. In Section
\ref{sec:classical-parameters} we correlate classical and spectral parameters
of a hypergraph, such as chromatic number and diameter, with spectral radius
and number of distinct eigenvalues. In Section~\ref{sec:power} we study the
spectrum of the signless Laplacian matrix of a power hypergraph.

\section{Preliminaries}\label{sec:pre}

In this section, we shall present some basic definitions about hypergraphs
and matrices, as well as terminology, notation and concepts that will be
useful in our proofs. More details about hypergraphs can be found in
\cite{Bretto}.

A \textit{hypergraph} $\h=(V,E)$ is a pair composed by a set of vertices
$V(\h)$ and a set of (hyper)edges $E(\h)\subseteq 2^V$, where $2^V$ is the
power set of $V$. $\h$ is said to be a $k$-\textit{uniform} (or a $k$-graph)
for $k \geq 2$, if all edges have cardinality $k$. Let $\mathcal{H}=(V,E)$
and $\mathcal{H}'=(V',E')$ be hypergraphs, if $V'\subseteq V$ and
$E'\subseteq E$, then $\mathcal{H}'$ is a \textit{subgraph} of $\h$.

The \textit{neighborhood} of a vertex $v\in V(\h)$, denoted by $N(v)$, is the
multi-set formed by all vertices, distinct from $v$, that have some edge in
common with $v$, where the multiplicity of each element in the multi-set is
exactly the number of edges in common with the vertex $v$. The \textit{edge
neighborhood} of a vertex $v\in V$, denoted by $E_{[v]}$, is the set of all
edges that contain $v$. More precisely, $\;E_{[v]}=\{e:\;v\in e\in E\}$.

The \textit{degree} of a vertex $v\in V$, denoted by $d(v)$, is the number of
edges that contain $v$. More precisely, $\;d(v) = |E_{[v]}|$. A hypergraph is
$r$-\textit{regular} if $d(v) = r$ for all $v \in V$. We define the
\textit{maximum}, \textit{minimum} and \textit{average} degrees,
respectively, as \[\Delta(\h) = \max_{v \in V}\{d(v)\}, \quad \delta(\h) =
\min_{v \in V}\{d(v)\}, \quad d(\h) = \frac{1}{n}\sum_{v \in V}d(v).\] When
we are working with more than one hypergraph, we can use the notation
$d_\h(v)$, to avoid ambiguity.

Let $\h$ be a hypergraph. A \textit{walk} of length $l$ is a sequence of
vertices and edges $v_0e_1v_1e_2 \ldots e_lv_l$ where $v_{i-1}$ and $v_i$ are
distinct vertices contained in $e_i$ for each $ i=1,\ldots,l$. The
\textit{distance} between two vertices is the length of the shortest walk
connecting these two vertices. The \textit{diameter} of the hypergraph is the
largest distance between two of its vertices. The hypergraph is
\textit{connected}, if for each pair of vertices $ u, w$ there is a walk
$v_0e_1v_1e_2 \cdots e_lv_l $ where $ u = v_0 $ and $ w = v_l $. Otherwise,
the hypergraph is \textit{disconnected}.

Let $\g$ and $\h$ be $k$-graphs. We define its \textit{union} $\g\cup\h$ as
the $k$-graph, with the sets of vertices  $\;V(\g\cup\h) = V(\g)\cup V(\h)$
and edges $\;E(\g\cup\h) = E(\g)\cup E(\h)$. The \textit{cartesian product}
$\g\times\h$ is the $k$-graph, with the sets of vertices $\;V(\g\times\h) =
V(\g)\times V(\h)$ and edges	$\;E(\g\times\h) = \{\{v\}\times e:\; v \in
V(\g),\; e \in E(\h)\}\cup \{a\times \{u\}:\; u \in V(\h),\; a \in E(\g)\}$.

A \textit{multigraph} is an ordered pair $\g = \left(V, E\right)$, where $V$
is a set of vertices and $E$ is a multi-set of pairs of distinct, unordered
vertices, called edges. Its \textit{adjacency matrix} $\A(\g)$, is the square
matrix of order $|V|$, where $a_{ii}=0$ and if $i \neq j$, then $a_{ij}$ is
the number of edges connecting the vertices $ i $ and $ j $.

Let $\M$ be a square matrix of order $n$. We denote its characteristic
polynomial by $P_\M(\lambda)=\det(\lambda\mathbf{I}_n-\M)$. Its eigenvalues
will be denoted by $\lambda_1(\M)\geq\cdots\geq\lambda_n(\M)$. If $\x$ is an
eigenvector from eigenvalue $\lambda$, then the pair $(\lambda,\x)$ will be
called \textit{eigenpair} of $\M$. The \textit{spectral radius} $\rho(\M)$,
is the largest modulus of an eigenvalue.

\section{Incidence matrix, clique and line multigraphs}\label{sec:incidence}

In this section, we will study the incidence matrix of a hypergraph. More
specifically, we will analyze the relationship of this matrix with two
multigraphs associated with it: the line and clique multigraphs. The results of this section are generalizations of well-known properties of the incidence matrix and line graphs \cite{Line2, Line}.

\begin{Def}	
Let $\h = (V, E)$ be a hypergraph. The \textit{incidence matrix} $\B(\h)$ is defined
as the matrix of order $|V|\times|E|$, where $\;b(v,e) = 1$ if  $v \in e$ and
$\;b(v,e) = 0$ otherwise. Its \textit{matrix of degrees} $\D(\h)$, is a square matrix
of order $|V|$, where $d_{ii} = d(i)$ and if $i \neq j$, then $d_{ij} = 0$.
\end{Def}

The \textit{clique multigraph} $\cli(\h)$, is obtained by transforming the
vertices of $\h$ in its vertices. The number of edges between two vertices of
this multigraph is equal the number of hyperedges containing them in $\h$.
The \textit{line multigraph} $\lin(\h)$, is obtained by transforming the hyperedges of
$\h$ in its vertices. The number of edges between two vertices of this
multigraph is equal the number of vertices in common in the two respective
hyperedges.

\begin{Exe}
The clique and line multigraphs from $\h=(\{1,\ldots,5\},
\;\{123,145,345\})$, are illustrate in Figure \ref{fig:multigrafos}.
		\begin{figure}[!h]		
		\begin{tikzpicture}
		\node[draw,circle,fill=black,label=below:,label=above:\(1\)] (v1) at (0,0) {};
		\node[draw,circle,fill=black,label=below:,label=above:\(2\)] (v2) at (1.5,2) {};
		\node[draw,circle,fill=black,label=below:,label=above:\(3\)] (v3) at (3,0) {};
		\node[draw,circle,fill=black,label=below:,label=above:\(4\)] (v4) at (1,0) {};
		\node[draw,circle,fill=black,label=below:,label=above:\(5\)] (v5) at (2,0) {};
		
		\begin{pgfonlayer}{background}
		\draw[edge,color=gray,line width=20pt] (v1) -- (v2) -- (v3);
		\draw[edge,color=gray,line width=25pt] (v5) -- (v1) -- (v4) -- (v5);
		\draw[edge,color=gray,line width=30pt] (v4) -- (v3) -- (v5) -- (v4);		
		\end{pgfonlayer}				
		\end{tikzpicture}
		\begin{tikzpicture}	
		\node[draw,circle,fill=black,label=below:,label=above:\(1\)] (v1) at (0,0) {};
		\node[draw,circle,fill=black,label=below:,label=above:\(2\)] (v2) at (1.5,2) {};
		\node[draw,circle,fill=black,label=below:,label=above:\(3\)] (v3) at (3,0) {};
		\node[draw,circle,fill=black,label=below:,label=above:\(4\)] (v4) at (1,0) {};
		\node[draw,circle,fill=black,label=below:,label=above:\(5\)] (v5) at (2,0) {};
		
		\path
		(v1) edge node[below]{} (v2)
		(v2) edge node[below]{} (v3)
		(v3) edge [bend left] node[below]{} (v1)
		(v1) edge node[below]{} (v4)
		(v1) edge [bend left] node[below]{} (v5)
		(v4) edge [bend left] node[below]{} (v5)
		(v4) edge [bend right] node[below]{} (v5)
		(v4) edge [bend right] node[below]{} (v3)
		(v3) edge node[below]{} (v5);
		\end{tikzpicture}
		\begin{tikzpicture}	
		\node[draw,circle,fill=white, label=above:\(123\)] (a) at (1.5, 2) {};
		\node[draw,circle,fill=white, label=above:\(145\)] (b) at (0, 0) {};
		\node[draw,circle,fill=white, label=above:\(345\)] (c) at (3, 0) {};

		\path
		(a) edge node[below]{} (b)
		(a) edge node[below]{} (c)
		(b) edge [bend left] node[below]{} (c)
		(b) edge [bend right] node[below]{} (c);
		\end{tikzpicture}
		\caption{Clique $\cli(\h)$ and line $\lin(\h)$ multigraphs.}\label{fig:multigrafos}	
	\end{figure}	
\end{Exe}

Our first result is the following observation. We believe it is worth
mentioning because it opens the possibility of studying hypergraphs from the
spectrum of multigraphs.

\begin{Teo}\label{teo:multigrafo}	
Let $\h$ be a $k$-graph, $\B$  its incidence matrix, $\D$ its degree matrix,
$\A_\lin$ and $\A_\cli$ the adjacency matrices of its line and clique
multigraphs, respectively. Then
	\[\B^T\B = k\mathbf{I}+\A_\lin, \quad \textrm{ and } \quad\; \B\B^T =\D+\A_\cli.\]
\end{Teo}
\begin{proof}
Let $\C = \B^T\B$. Note that $c_{ij}$ is the number of vertices in common
between the hyperedges $e_i$ and $e_j$. So, if $i \neq j$, then $c_{ij}$ is
the number of edges between the vertices $i$ and $j$ in the line multigraph
$\lin(\h)$, otherwise $c_{ii} = k$.	Therefore, we conclude $\C =
k\mathbf{I}+\A_\lin$.

Now, let $\mathbf{M} = \B\B^T$. Note that $m_{ij}$ is the number of
hyperedges that contains at the same time the vertices $i$ and $j$. So we
have $m_{ii} = d(i)$ for all $i \in V$, and if $i\neq j$, then $m_{ij}$ is
the number of edges between the vertices $i$ and $j$ in the clique multigraph
$\cli(\h)$. Therefore, we conclude $\mathbf{M} = \D+\A_\cli$.
\end{proof}

\begin{Pro}\label{pro:polinomios}
If $\h$ is a $k$-graph, $r$-regular, with $n$ vertices and $m$ edges, then
	\[P_{\A_\lin}(\lambda)=(\lambda+k)^{m-n}P_{\A_\cli}(\lambda-r+k).\]
\end{Pro}
\begin{proof}
Let $\B$ be the incidence matrix of $\h$. Consider the following matrices.
\[U = \left[ \begin{smallmatrix}\lambda\mathbf{I}_n & -\mathbf{B}\\ \mathbf{0} &
\mathbf{I}_m \end{smallmatrix}\right], \;\; V = \left[ \begin{smallmatrix}
\mathbf{I}_n & \mathbf{B}\\ \mathbf{B}^T & \lambda\mathbf{I}_m \end{smallmatrix}
\right] \;\Rightarrow\; UV = \left[ \begin{smallmatrix}\lambda\mathbf{I}_n
-\mathbf{B}\mathbf{B}^T & \mathbf{0}\\ \mathbf{B}^T & \lambda\mathbf{I}_m \end{smallmatrix}
\right], \;\; VU = \left[ \begin{smallmatrix}\lambda\mathbf{I}_n & \mathbf{0}\\
\lambda\mathbf{B}^T & \lambda\mathbf{I}_m - \mathbf{B}^T\mathbf{B}
\end{smallmatrix}\right].\]
We know that $det(VU) = det(UV)$. So,
\begin{equation}\label{teo:transp}
\lambda^ndet(\lambda\mathbf{I}_m - \mathbf{B}^T\mathbf{B}) =
\lambda^{m}\det(\lambda\mathbf{I}_n-\mathbf{B}\mathbf{B}^T).	
\end{equation}	
Thus,
\begin{eqnarray*}
P_{\A_\lin}(\lambda) &=& det(\lambda\mathbf{I}_m - \A_\lin) = det((\lambda+k)\mathbf{I}_m - \B^T\B)\\
					 &=&(\lambda+k)^{m-n}det((\lambda+k)\mathbf{I}_n - \B\B^T)\\
					 &=&  (\lambda+k)^{m-n}det((\lambda+k-r)\mathbf{I}_n - \A_\cli)\\
					 &=& (\lambda+k)^{m-n}P_{\A_\cli}(\lambda-r+k).\\
\end{eqnarray*}
Therefore, the result follows.
\end{proof}

\begin{Lem}\label{lem:graulinha}
Let $\h$ be a $k$-graph and $\lin(\h)$ its line graph. If $u \in
V(\lin(\h))$ is a vertex obtained from the edge $e\in E(\h)$, then
$$d_\lin(u) = \left( \sum_{v \in e} d_\h(v)\right)-k.$$
\end{Lem}
\begin{proof}
Notice that, for each $v \in e$, there exist other $d_\h(v)-1$ hyperedges
containing it. That is, this vertex will generate $d_\h(v)-1$ edges
containing $u$ in $\lin(\h)$. Using the same argument for the other vertices
of $e$, we conclude that the degree of the vertex $u$ in the line multigraph,
must be $\;d_\lin(u) = \sum_{v \in e}\left( d_\h(v) -1\right)$.
\end{proof}

\section{Signless Laplacian matrix}\label{sec:Laplacian}

In this section, we study some properties of the signless Laplacian matrix of
a hypergraph, generalizing important results of this matrix in the context of
spectral graph theory, whose main results may be found in the series of papers
by Cvetkovi\'c, Rowlinson and Simi\'c \cite{Towards0, TowardsI, TowardsII,
TowardsIII}, and references therein.

\begin{Def}
Let $\h$ a hypergraph and $\B$ its incidence matrix. The \textit{signless
Laplacian matrix} is defined as $\Q(\h) = \B\B^T$.
\end{Def}

An \textit{oriented hypergraph} $\mathit{H}=(\h,\sigma)$ is a hypergraph
where for each vertex-edge incidence $(v,e)$ it is given a label $\sigma(v,e)
\in \{+1,-1\}$. In \cite{Reff2012}, Reef and Rusnak define the incidence
matrix of an oriented hypergraph $\mathfrak{B}(\mathit{H})$ by
$\;\mathfrak{b}(v,e) = \sigma(v,e)$ if  $v \in e$ and $\;\mathfrak{b}(v,e) =
0$ otherwise. The \textit{Laplacian matrix} for oriented hypergraphs is
defined as, $\mathfrak{L}(\mathit{H})= \mathfrak{B}\mathfrak{B}^T$. We
observe that if $\sigma(v,e) = 1$  for all vertex-edge incidence $(v,e)$,
then this definition coincides with the our definition of signless Laplacian
matrix.

\begin{Obs}
Let $\h$ be a $k$-graph, and $\Q$ its signless Laplacian matrix. This matrix
has some simple but useful properties, such as being symmetric, non-negative
and positive semi-definite. Further, if $\h$ is connected, then $\Q$ is
irreducible. These properties allow us to conclude directly from matrix
theory the Rayleigh principle and Perron-Frobenius Theorem, stated bellow.
\end{Obs}

\begin{Teo}[Rayleigh principle for hypergraphs]\label{teoray}
Let $\h$ be a $k$-graph. If $\lambda_1$ is the largest eigenvalue of $\Q$,
then
	\[\lambda_1 = \rho(\Q) = \max_{||\x||=1}\{\x^T\Q\x\}\geq 0.\] Further,
the equality is achieved if and only if $\x$ is an eigenvector of
$\lambda_1$.
\end{Teo}
\begin{Teo}[Perron-Frobenius Theorem for hypergraphs]\label{perron}
Let $\h$ be a $k$-graph. If $\h$ is connected, then $\rho(\Q)$ is an
algebraically simple eigenvalue, with a positive eigenvector.
\end{Teo}

The vector obtained in the Perron-Frobenius Theorem, normalized under norm
$\ell^2$, is referred to as \emph{principal eigenvector} of $\Q(\h)$. Some
times we will denote $\rho(\Q) = \rho(\h)$.

We finish this section proving some basic properties of the signless
Laplacian matrix for uniform hypergraphs.

\begin{Lem}
Let $\h$ be a $k$-graph and $\Q=(q_{ij})$ its signless Laplacian matrix.
Then, for each $i \in V$, we have $$\sum_{j \in V}q_{ij} = kd(i).$$
\end{Lem}
\begin{proof}
By the characterization of signless Laplacian matrix of Theorem
\ref{teo:multigrafo}, we have  $$\sum_{j \in V}q_{ij}\, =\, d(i)+\sum_{j \in
N(i)}1\, =\, d(i)+(k-1)d(i)\, =\, kd(i).$$
\end{proof}

\begin{Pro}\label{pro:poliqpolilin}
	If $\h$ is a $k$-graph with $n$ vertices and $m$ edges, then
	$$ P_{\A_\lin}(\lambda) = (\lambda+k)^{m-n}P_\Q(\lambda+k).$$
\end{Pro}
\begin{proof}
	Notice that, by equation (\ref{teo:transp}), we have
	\begin{eqnarray}
	\notag P_{\A_\lin}(\lambda) &=& \det(\lambda \mathbf{I}_m-\A_\lin) = \det((\lambda+k)\mathbf{I}_m-\B\B^T)\\
	 \notag &=& (\lambda+k)^{m-n}\det((\lambda+k)\mathbf{I}_n-\B^T\B) = (\lambda+k)^{m-n}P_\Q(\lambda+k).
	 \end{eqnarray}
	Therefore, the result follows.	
\end{proof}
\begin{Obs}
Here we highlight two interesting consequences of the Proposition
\ref{pro:poliqpolilin}. First, if $\lambda$ is an eigenvalue of $\A_\lin$,
then $\lambda \geq -k$. Second, we see that $\rho(\A_\lin) = \rho(\Q) - k$.
\end{Obs}

\begin{Pro}
	Let $\g$ and $\g'$ be two $k$-graphs. If $\h=\g\cup\g'$, then
	\[Spec(\Q(\h)) = Spec(\Q(\g))\cup Spec(\Q(\g')).\]
\end{Pro}
\begin{proof}
	If we first enumerate all the vertices in $ \g $ and then the vertices in $ \g' $, the signless Laplacian matrix of $ \h $, will have the following form $\;\Q(\h) = \left(\begin{smallmatrix}\Q(\g)&\mathbf{0}\\\mathbf{0}&\Q(\g')\end{smallmatrix} \right)$.
	Thus, $det(\lambda \mathbf{I}-\Q(\h))=det(\lambda \mathbf{I}-\Q(\g))det(\lambda \mathbf{I}-\Q(\g'))$.
	So, the result follows.
\end{proof}

We now introduce the following notation. Let $\h=(V,E)$ be a hypergraph, for each non-empty subset of
vertices $\alpha = \{v_1,\ldots,v_t\} \subset V$, given a vector $\x=(x_i)$ of dimension $n=|V|$, we denote $x(\alpha)=x_{v_1}+\cdots+x_{v_t}$. Under these conditions we can write,
\[(\Q\x)_u = d(u)x_u+\sum_{w \in N(u)}x_w = \sum_{e \in E_{[u]}}x(e).\]

\begin{Pro}
If $\g$ and $\h$ are two $k$-graphs, with signless Laplacian eigenvalues
$\mu$ of multiplicity $m_1$ and  $\lambda$ of multiplicity $m_2$
respectively, then $\mu+\lambda$ is an eigenvalue of $\Q(\g\times\h)$, with
multiplicity $m_1\cdot m_2$.
\end{Pro}
\begin{proof}
Suppose $\x$ an eigenvector of $\lambda$ in $\Q(\h)$ and  $\y$ an eigenvector
of $\mu$ in $\Q(\g)$. Consider $(v,u)$ a vertex of $\g\times\h$, define a
vector $\mathbf{z}$ by $z_{(v,u)} = y_vx_u$. Thus,
	\[(\Q\mathbf{z})_{(v,u)}=\!\!\!\sum_{\alpha \in E_{[(v,u)]}}\!\!\!z(\alpha) = \sum_{e \in E_{[u]}}y_vx(e) + \sum_{a \in E_{[v]}}y(e)x_u = \lambda y_vx_u + \mu y_vx_u = (\mu+\lambda)z_{(v,u)}.\]
	Therefore, the result is true.
\end{proof}

The result bellow, may be seen as a corollary of Proposition 4.4 of \cite{Reff2014}.
\begin{Pro}\label{pro:xqx}
Let $\h$ be a $k$-graph with $n$ vertices. For each vector $\x\in\mathbb{R}^n$, we have
	\[\x^T\Q\x=\sum_{e \in E}\left[ x(e)\right]^2.\]
\end{Pro}
\begin{proof}
Notice that, for each edge $e\in E$ is true that $(\B^T\x)_e = x(e)$.
Therefore,
	\[\x^T\Q\x = \x^T(\B\B^T)\x = (\B^T\x)^T(\B^T\x) = \sum_{e \in E}\left[x(e)\right]^2.\]
\end{proof}

\begin{Pro}
Let $\h$ be a connected $k$-graph. If $\h'$ is a subgraph of $\h$, then
$$\rho(\h') \leq \rho(\h).$$
\end{Pro}
\begin{proof}
Let $\x$ be the principal eigenvector of $\Q(\h')$. Define a new vector
$\overline{\x}$ of dimension $n=|V(\h)|$, by $\overline{x}_i=x_i$ if $i \in
V(\h')$ and $\overline{x}_i=0$ otherwise. Thus,
	\[\rho(\h)\geq \sum_{e \in E(\h)} [\overline{x}(e)]^2 = \sum_{e \in E(\h')} [x(e)]^2 = \rho(\h').\]
\end{proof}

\section{Structural and spectral properties}\label{sec:strutural-properties}

In this section, we will determine structural characteristics of a hypergraph
from its signless Laplacian spectrum. More precisely, we will study
regular and partially bipartite uniform hypergraphs through their signless
Laplacian eigenvalues.

\begin{Teo}\label{teo:qregular}
Let $\h$ be a connected $k$-graph. The following statements are equivalent
\begin{itemize}
	\item[(a)] $\h$ is regular.
	\item[(b)] $\rho(\h) = kd(\h)$.
	\item[(c)] $\rho(\h) = k\Delta(\h)$.
	\item[(d)] The principal eigenvector of $\Q(\h)$ is $\x=\left(\frac{1}{\sqrt{n}},\ldots,\frac{1}{\sqrt{n}}\right)$, where $n = |V(\h)|$.
\end{itemize}
\end{Teo}
\begin{proof}
We will prove the result through the following chain of implications,
	$$(a) \Rightarrow (b) \Rightarrow (d) \Rightarrow (c) \Rightarrow (a).$$
	
Suppose $\h$ is $r$-regular, then for each vertex $u$ we have $|E_{[u]}|=r$. Thus,
	\[(\Q\mathbf{1})_u = \sum_{e \in E_{[u]}}x(e) = kr(1).\]

\noindent That is, $\mathbf{1}$ is an eigenvector associated with the
eigenvalue $kr$, and since $\h$ is regular, then $r = d(\h)$. By
Perron-Frobenius Theorem \ref{perron}, we conclude that $\rho(\h) = kd(\h)$. 	

Now, suppose $\rho(\h) = kd(\h)$. We notice that the vector
$\x=\left(\frac{1}{\sqrt{n}},\ldots,\frac{1}{\sqrt{n}}\right)$ solves the
following optimization problem.
	$$\rho(\h) = \max_{||\y||=1}\{\y^T\Q\y\}\geq \x^T\Q\x = \sum_{i \in V}
\frac{kd(i)}{n} = kd(\h).$$

\noindent Thus, by Rayleigh principle \ref{teoray}, we conclude that $\x$ is
the principal eigenvector of $\Q$. 	

Let $\x=\left(\frac{1}{\sqrt{n}},\ldots,\frac{1}{\sqrt{n}}\right)$ be the
principal eigenvector of $\Q$. If $u \in V$ is a vertex of maximum degree,
then
	\[\rho(\h)\left(\frac{1}{\sqrt{n}}\right) = (\Q\mathbf{x})_u = \sum_{e \in E_{[u]}}x(e) = \Delta k\left(\frac{1}{\sqrt{n}}\right), \quad \Rightarrow \quad  \rho(\h) = k\Delta(\h).\]

If $\rho(\h) = k\Delta(\h)$ is the spectral radius of $\Q$ and $\x$ its
principal eigenvector. Let $u \in V$ be a vertex, such that $x_u \geq x_v$
for all $v \in V$. Thus,
$$k\Delta x_{u} = \sum_{e \in E{[u]}}x(e).$$

\noindent We observe that this equality is only possible if, $d(u) = \Delta$
and $x_v = x_{u}$ for all $v \in N(u)$. Hence, we conclude that every vertex
that has maximum value in the eigenvector $\x$, must have maximum degree.
Moreover, every vertex that is neighbor of another vertex that has maximum
value in the eigenvector, must also have maximum value. By the connectivity
of the hypergraph, we conclude that all the vertices have maximum value in
the principal eigenvector and therefore maximum degree, i.e. $\h$ is regular.
\end{proof}

\begin{Lem}\label{lem:caracZerox}
	Let $\h$ be a $k$-graph. Thus, $(0,\x)$ is an eigenpair of $\Q(\h)$ if, and only if,
	for each edge $e\in E$ we have $\;x(e)=0.$
\end{Lem}
\begin{proof}
If $(0,\x)$ is a signless Laplacian eigenpair of $\h$, then $\Q\x =
\mathbf{0}$. So,
	$$\x^T\Q\x = \x^T\mathbf{0} = 0 \quad\Rightarrow\quad \sum_{e \in E}[x(e)]^2 = 0\quad \Rightarrow\quad x(e) = 0,\;\; \forall e \in E.$$
	
Conversely, let $\x$ be a vector of dimension $n=|V(\h)|$, such that $x(e)
= 0$, for each edge $e \in E$. So,
	$$(\Q\x)_u = \sum_{e \in E_{[u]}}x(e) = 0,\;\; \forall u  \in V \quad\Rightarrow\quad (0,\x)\; \textrm{ is an eigenpair of } \Q(\h).$$	
\end{proof}

We notice that for a graph, the condition $\;x_{v_i}+x_{v_j}=0$ for all
$e=\{v_i,v_j\}\in E$, implies a bipartition of vertices. Unfortunately for $k
\geq 3$, we do not have such a trivial characterization.
\begin{Def}
A hypergraph $\h$ is \textit{partially bipartite}, if we can separate the set
of vertices into three disjoint subsets $V = V_0\cup V_1\cup V_2$, where
$V_1$ and $V_2$ are non empty and each edge is fully contained in $V_0$ or
has vertices in both $V_1$ and $V_2$.
\end{Def}

Now, we prove Theorem \ref{teo:caracZero}. We state here again for easy
reference.

\noindent\textbf{Theorem \ref{teo:caracZero}. }\textit{Let $\h = (V,E)$ be a
$k$-graph. If $\lambda=0$ is an eigenvalue of $\Q$, then $\h$ is partially
bipartite.}
\begin{proof}
Let $\x$ be a eigenvector of $\lambda=0$. Define $$V_1 = \{v \in V: x_v >
0\},\; V_2 = \{v \in V: x_v < 0\}, \; V_0 = \{v \in V: x_v = 0\}.$$ As $x(e)
= 0$ for each edge $e \in E$, then the edge is contained in $V_0$, or it must
to have some vertices in $V_1$ and others in $V_2$, i.e. $\h$ is partially
bipartite. 		
\end{proof}
\begin{Def}
A hypergraph $\h$ is \textit{balanced partially bipartite}, if it is
partially bipartite and there exists a constant $c>0$, such that for each
edge $e \nsubseteq V_0$, it happens $\frac{|e\cap V_1|}{|e\cap V_2|} = c$.
\end{Def}
\begin{Teo}\label{cor:caracZero2}
	Let $\h = (V,E)$ be a $k$-graph. If it is balanced partially bipartite, then $\lambda=0$ is an eigenvalue of $\Q(\h)$.
\end{Teo}

\begin{proof}
	Since $\h$ is balanced partially bipartite, there is a constant $c = \frac{|e\cap V_1|}{|e\cap V_2|}$, where $e \in E(\h)$. So we define a vector $\x$ of dimension $n = |V|$, by
	$$x_v = \begin{cases} \quad 1 \quad \textrm{if}\;  v \in V_1 \\ \quad 0 \quad \textrm{if}\;  v \in V_0 \\ \; -c \quad \textrm{if}\;  v \in V_2 \\ \end{cases}$$
	Thus $x(e)=0$, for each edge $e \in E$. By Lemma \ref{lem:caracZerox}, we conclude the result.
\end{proof}

\begin{Exe} We illustrate here that balanced partially bipartite graphs
are abundant, by presenting two examples that are easy to find. Let $\h =
(V,E)$ be a $k$-graph in which 	
	\begin{enumerate}
\item  $\h$ has a vertex $v$ with the property that each edge containing it
    also contains another vertex of degree 1;
\item $\h$ has a couple of vertices which are contained in exactly the same
    edges;
	\end{enumerate} In both cases $\h$ is balanced partially bipartite.
\end{Exe}

\section{Relating classical and spectral parameters}
\label{sec:classical-parameters}

In this section, we will relate classic and spectral parameters of a
hypergraph. More precisely, we will relate the spectral radius to the degrees
and the chromatic number, the number of edges is related to the sum of the
eigenvalues, and the diameter is related to the number of distinct
eigenvalues of the signless Laplacian matrix.

\begin{Teo}\label{teo:qcotas}
If $\h$ is a connected $k$-graph and $\rho(\Q)$ is its spectral radius, then
	\[\min_{e\in E}\left\lbrace \sum_{v \in e}d(v)\right\rbrace  \leq \rho(\Q) \leq \max_{e\in E}\left\lbrace \sum_{v \in e}d(v)\right\rbrace.\]
\end{Teo}
\begin{proof}
For each $u\in V(\lin(\h))$, suppose $e_u \in E(\h)$ is the edge that gives
vertex $u$. By Lemma \ref{lem:graulinha}, we have $\;d_{\lin}(u)\, =\, \left(
\sum_{v \in e_u}d_{\h}(v)\right) -k.$ Now, by Theorem \ref{pro:poliqpolilin},
we have $\rho(\A_\lin) = \rho(\Q) - k$. For graphs and multigraphs, we know
$$\min_{ u\in V(\lin(\h))}d_\lin(u) \leq \rho(\A_\lin) \leq \max_{ u\in
V(\lin(\h))}d_\lin(u).$$
	Therefore,
	\[\min_{\{v_1,\ldots,v_k\}\in E}\left\lbrace \left(
\sum_{i=1}^kd(v_i)\right) -k\right\rbrace  \leq \rho(\Q) -k\leq
\max_{\{v_1,\ldots,v_k\}\in E}\left\lbrace \left(\sum_{i=1}^k d(v_i)\right)
-k\right\rbrace.\]	
Adding $k$ in each of the three parts of the
inequalities, we obtain the desired result.
\end{proof}

The result below may be seen as a corollary of Propositions 4.7 and 4.12 in
\cite{Reff2014}.
\begin{Cor}
If $\h$ is a connected $k$-graph and $\rho(\h)$ is its spectral radius, then
	\[kd(\h) \leq \rho(\h) \leq k\Delta(\h).\]
\end{Cor}
\begin{proof}
If $n=|V|$, define
$\x=\left(\frac{1}{\sqrt{n}},\ldots,\frac{1}{\sqrt{n}}\right)$. By Theorems
\ref{teoray} and \ref{teo:qcotas}, we have
	$$kd(\h) = \x^T\Q\x \leq \rho(\h) \leq \max_{e\in E}\left\lbrace \sum_{v \in e}d(v)\right\rbrace \leq k\Delta(\h).$$	
\end{proof}

\begin{Def}
For a $k$-graph $\h$, a function $f:V \rightarrow \{1,\ldots,r\}$ is a
(vertex) $r$-\textit{coloring} of $\h$, if for every edge $e =
\{v_1,\ldots,v_k\}$ there exists $i \neq j$ such that $f(v_i) \neq f(v_j)$.
The \textit{chromatic number} $\chi(\h)$, is the minimum integer $r$ such
that $\h$ has an $r$-coloring.
\end{Def}

\begin{Teo}
Let $\h$ be a connected $k$-graph. If $\chi(\h)$ is its chromatic number,
then $$\chi(\h)\leq \frac{1}{k}\rho(\h)+1.$$
\end{Teo}

\begin{proof}
We will define an order for the vertices of $\h$ as follows. Let $ \h(n) =
\h $ and $ v_n $ be a vertex of minimum degree in $\h(n)$. For each $t=2,\ldots,n$, let $
\h(t-1) $ be the subgraph obtained after removing a vertex, $v_{t}$ with
minimum degree from $ \h(t)$.

Let us use the ordering $ v_1, v_2, \ldots, v_n $, as input of a greedy
coloring algorithm, which paints $ v_t $ with the smallest color that makes $
\h(t) $ properly colored. 	

Notice that $\chi(\h(1))=1\leq \frac{1}{k}\rho(\h)+1$. Inductively, suppose $
\h(t-1) $ is properly colored with up to $ \frac{1}{k} \rho(\h) + 1 $
distinct colors. We see that $ v_t $ has a minimum degree in $ \h(t) $. Thus,
in the worst case, each edge containing $ v_t $ has all the other vertices
painted with the same color, and each of these edges uses one of the colors $
1,2, \ldots, d_{\h(t)}(v_t)$. So we should paint $ v_t $ with the color
$d_{\h(t)}(v_t)+1$. Thus,
	$$d_{\h(t)}(v_t)+1 = \delta(\h(t))+1\leq \frac{1}{k}\rho(\h(t))+1\leq
\frac{1}{k}\rho(\h)+1.$$
By induction hypothesis, we have $\chi(\h(t))\leq \frac{1}{k}\rho(\h)+1$. So, $\chi(\h)\leq \frac{1}{k}\rho(\h)+1$.
\end{proof}

\begin{Pro}
Let $\h$ be a $k$-graph with characteristic polynomial $P_\Q(\lambda) =
\lambda^n+q_1\lambda^{n-1}+\cdots+q_{n-1}\lambda+q_n$.  The number of edges of
of $\h$ is given by $m = -\frac{q_1}{k}$.
\end{Pro}
\begin{proof}
If $\lambda_1\geq\lambda_2\geq\cdots\geq\lambda_n$ are all eigenvalues of the
matrix $\Q$, then $$q_1 = -(\lambda_1+\cdots+\lambda_n)= -Tr(\Q) = -km.$$
\end{proof}

\begin{Teo}
Let $\h$ be a $k$-graph with diameter $D$. The number of distinct eigenvalues
of the matrix $\Q$ is at least $D+1$.
\end{Teo}
\begin{proof}	
First we will show the following claim.
\begin{Afi}
If there is a walk with length $l$ connecting two distinct vertices $i$ and
$j$, then $\;(\Q^l)_{ij}>0,\;$  otherwise $\;(\Q^l)_{ij} = 0$.
\end{Afi}
The proof is by induction on $l$. We first notice that if $l = 1$ then the
signless Laplacian matrix has the desired properties. Now suppose the
statement is true for $l \geq 1$. Note that,
$$(\Q^{l+1})_{ij} = \sum_{t=1}^n(\Q^{l})_{it}(\Q)_{tj}.$$

\noindent Thus, if there is no walk with length $l + 1$, linking $ i $ and $
j $, then there can be no walk linking $ i $ to a neighbor of $j$. This
implies that, if $u$ is a neighbor of $j$ then $(\Q^{l})_{iu}=0$ and
otherwise $(\Q)_{uj}=0$. Therefore, $(\Q^{l+1})_{ij} = 0$. On the other hand,
assuming there is a walk with length $l + 1$, linking $ i $ and $ j $, then
there must be a walk  with length $l$, linking $ i $ to a neighbor $u$ of
$j$. So, $(\Q^{l})_{iu}>0$ and $(\Q)_{uj}>0$. Therefore, $(\Q^{l+1})_{ij}
> 0$. The claim is proven.

Returning to the proof of the theorem, we let $\lambda_1, \lambda_2 \ldots,
\lambda_t$ be all the distinct eigenvalues of $\Q$. So, $(\Q-\lambda_1
\mathbf{I})(\Q-\lambda_2 \mathbf{I})\cdots(\Q-\lambda_t \mathbf{I}) = 0.$
Thus, 	$ \Q^t+a_1\Q^{t-1}+\cdots+a_t\mathbf{I} = 0.$ Suppose, by way of
contradiction that $D \geq t$. Hence, there must exist $i$ and $j$ such that
its distance is $t$. Thus, $(\Q^t)_{ij} =
-a_1(\Q^{t-1})_{ij}-\cdots-a_t(\mathbf{I})_{ij} = 0$, because there should be
no walk shorter than $t$ linking the vertices $i$ and $j$. This contradicts
the claim. Therefore $t \geq D + 1$.
\end{proof}
\begin{Teo}
Let $\h$ be a $k$-graph with more than one edge. If the diameter of $\h$ is
$D$, then
		$$D \leq \left\lfloor 1+
\frac{\log((1-x_{min}^2)/x_{min}^2)}{\log(\lambda_1/\lambda_2)}
\right\rfloor,$$ where $\lambda_1>\lambda_2$ are the greatest eigenvalues of
$\Q$ and $x_{min}$ is the smallest entry of the principal eigenvector.
\end{Teo}
\begin{proof}
As $\Q$ is real symmetric we may consider the orthonormal eigenvectors
$\x_1,\ldots,\x_n$ from the eigenvalues
$\lambda_1>\lambda_2\geq\cdots\geq\lambda_n$, respectively. In this case,
$\x_1$ is the principal eigenvector. Let $i$ and $j$ be vertices such that,
its distance is $D$. Using the spectral decomposition of $\Q$, for each
integer $t$, we have
	\begin{eqnarray}
\notag	(\Q^t)_{ij} &=& \sum_{l=1}^n \lambda_l^t(\x_l\x_l^T)_{ij} \geq \lambda_1^t(\x_1)_i(\x_1)_j - \left|\sum_{l=2}^n\lambda_l^t(\x_l\x_l^T)_{ij} \right|\\ \notag
	            &\geq& \lambda_1^tx_{min}^2 - \lambda_2^t\left(\sum_{l=2}^n(\x_l)_i^2 \right)^{\frac{1}{2}}  \left(\sum_{l=2}^n(\x_l)_j^2 \right)^{\frac{1}{2}} \\ \notag
	            &\geq& \lambda_1^tx_{min}^2 - \lambda_2^t\left(1-(\x_1)_i^2\right)^{\frac{1}{2}} \left(1-(\x_1)_j^2\right)^{\frac{1}{2}}\\
	            &\geq& \lambda_1^tx_{min}^2 - \lambda_2^t\left(1-x_{min}^2\right).  \label{exp}
	\end{eqnarray}
Notice that, if the expression (\ref{exp}) is positive, then  $(\Q^t)_{ij}$ is
positive and therefore $t \geq D$.
$$\lambda_1^tx_{min}^2 - \lambda_2^t\left(1-x_{min}^2\right) > 0\quad \Rightarrow \quad t > \frac{\log((1-x_{min}^2)/x_{min}^2)}{\log(\lambda_1/\lambda_2)}.$$
Therefore, the result follows.
\end{proof}

\section{Power hypergraph}\label{sec:power}
In this section, we will study the spectrum of the class of power
hypergraphs, relating its signless Laplacian eigenvalues to those of its base
hypergraph. The spectrum of this class has already been studied in the
context of tensors. See for example \cite{Kaue, Hu}.

\begin{Def}
Let $\h=(V,E)$ be a $k$-graph, let $s \geq 1$ and $r \geq ks$ be
integers. We define the (generalized) \textit{power hypergraph} $\h^r_s$ as
the $r$-graph with the following sets of vertices and edges
$$V(\h^r_s)=\left( \bigcup_{v\in V} \varsigma_v\right) \cup \left(
\bigcup_{e\in E} \varsigma_e\right)\;\; \textrm{and}\;\;
E(\h^r_s)=\{\varsigma_e\cup \varsigma_{v_1} \cup \cdots \cup \varsigma_{v_k}
\colon e=\{v_1,\ldots, v_k\} \in E\},$$	where $\varsigma_{v}=\{v_{1}, \ldots,
v_{s}\}$ for each vertex $v \in V(\h)$ and $\varsigma_e=\{v^1_e,
\ldots,v^{r-ks}_e\}$ for each edge $e \in E(\h)$.
\end{Def}

Informally, we say that $\h^r_s$ is obtained from a \textit{base hypergraph}
$\h=(V,E)$, by replacing each vertex $v\in V$ by a set $\varsigma_v$ of
cardinality $s$, and by adding a set $\varsigma_e$ with $r-ks$ new vertices,
to each edge $e \in\h$.

\begin{Exe}
The power hypergraph $(P_4)^5_2$ of the path $P_4$ is illustrated in
Figure \ref{fig:ex1}.
	\begin{figure}[h]
		\centering
		\begin{tikzpicture}
		[scale=1,auto=left,every node/.style={circle,scale=0.9}]
		\node[draw,circle,fill=black,label=below:,label=above:\(v_1\)] (v1) at (0,0) {};
		\node[draw,circle,fill=black,label=below:,label=above:\(v_2\)] (v2) at (4,0) {};
		\node[draw,circle,fill=black,label=below:,label=above:\(v_3\)] (v3) at (8,0) {};
		\node[draw,circle,fill=black,label=below:,label=above:\(v_4\)] (v4) at (12,0) {};
		\path
		(v1) edge node[left]{} (v2)
		(v2) edge node[below]{} (v3)
		(v3) edge node[left]{} (v4);
		\end{tikzpicture}
		
			\begin{tikzpicture}
		\node[draw,circle,fill=black,label=below:,label=above:\(v_{11}\)] (v1) at (0,0) {};
		\node[draw,circle,fill=black,label=below:,label=above:\(v_{12}\)] (v11) at (1,0) {};
		\node[draw,circle,fill=black,label=below:,label=above:] (v22) at (2.25,0) {};

		\node[draw,circle,fill=black,label=below:,label=above:\(v_{21}\)] (v3) at (3.5,0) {};
		\node[draw,circle,fill=black,label=below:,label=above:\(v_{22}\)] (v31) at (4.5,0) {};
		\node[draw,circle,fill=black,label=below:,label=above:] (v42) at (5.75,0) {};
	
		\node[draw,circle,fill=black,label=below:,label=above:\(v_{31}\)] (v5) at (7.5,0) {};
		\node[draw,circle,fill=black,label=below:,label=above:\(v_{32}\)] (v51) at (8.5,0) {};
		\node[draw,circle,fill=black,label=below:,label=above:] (v62) at (9.75,0) {};
		
		\node[draw,circle,fill=black,label=below:,label=above:\(v_{41}\)] (v64) at (11,0) {};
		\node[draw,circle,fill=black,label=below:,label=above:\(v_{42}\)] (v7) at (12,0) {};

		\begin{pgfonlayer}{background}
		\draw[edge,color=gray] (v1) -- (v31);
		\draw[edge,color=gray,line width=25pt] (v3) -- (v51);
		\draw[edge,color=gray] (v5) -- (v7);
		\end{pgfonlayer}
		\end{tikzpicture}
		\caption{The power hypergraph $(P_4)^5_2$.}
		\label{fig:ex3}\label{fig:ex1}
	\end{figure}
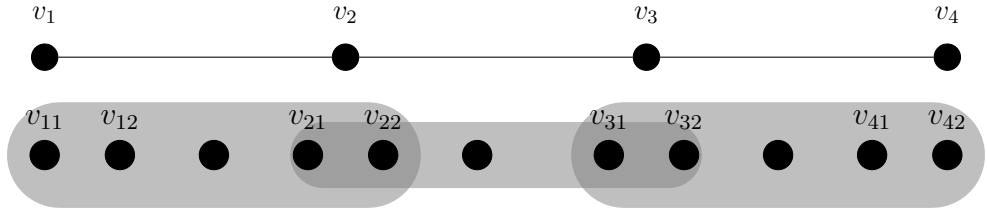
\end{Exe}

Let $\h^r_s$ be a power hypergraph. For each edge $e=\{i_1,\ldots,i_k\}\in
E(\h)$, we denote by $e^r_s=\varsigma_{i_1}\cup\cdots\cup \varsigma_{i_k}\cup
\varsigma_e \in E(\h^r_s)$ the edge obtained from $e \in E(\h)$. For
simplicity, we will write $\h^r = \h^r_1$ and $\h_s = \h^{ks}_s$. We identify
a vertex in each of the sets $\varsigma_v$ with the vertex $v$, and say that
it is a \emph{main vertice} of $\h^k_s$, while the other vertices in
$\varsigma_v$ are called \emph{copies}. The vertices in some of the sets
$\varsigma_e$ will be called \emph{additional vertices}.

We start this section by proving some algebraic properties of this class.

\begin{Lem}\label{teoadc}
Let $\h$ be a $k$-graph having two vertices $u$ and $v$ which are contained
exactly in the same edges. If $(\lambda, \x)$ is an eigenpair of $\Q$ with
$\lambda >0$, then $x_u = x_v$.
\end{Lem}
\begin{proof}
We just notice that,
	\[\lambda x_u = \sum_{e \in E_{[u]}}x(e) = \sum_{e \in E_{[v]}}x(e) =
\lambda x_v.\]
	Since $ \lambda \neq 0 $, then the result is true.
\end{proof}

\begin{Lem}\label{lem:verticeAdc}
Let $\h$ be a $k$-graph and $r \geq k$ be an integer. Then, for each eigenpair
$(\lambda,\x)$ of $\Q(\h^r)$, with $\lambda > r-k$, we have,
	\[x_{u} = \frac{x(e)}{\lambda-r+k},\quad \textrm{if } u
\textrm{ is an additional vertex of edge }e^r\in E(\h^r).\]
\end{Lem}
\begin{proof}
Suppose $\varsigma_e=\{i_e^1,\ldots,i_e^{r-k}\}$ and denote $u = i_e^1$. By Lemma
 \ref{teoadc} we know that $x_{u}=x_{i_e^{2}}\cdots=x_{i_e^{r-k}}$. So,
 $\;\lambda x_{u} = x(e^r)=(r-k)x_{u}+x(e)$. Thus,
	\[(\lambda-r+k)x_{u}=x(e) \quad \Rightarrow \quad x_{u} =
\frac{x(e)}{\lambda-r+k}.\]
\end{proof}

\begin{Pro}
Let $\h$ be a $k$-graph. If $\mu\neq0$ is a signless Laplacian eigenvalue of
$\h$, then $\lambda= \mu+r-k$ is an  eigenvalue of $\Q(\h^r)$.
\end{Pro}
\begin{proof}
Suppose $\y$ is an eigenvector of $\Q(\h)$, associated with $ \mu $. Define a
vector $ \x $ of dimension $|V(\h^r)|$, by
	\[x_i= \begin{cases}\;y_i \quad\, \textrm{ if }  i
\textrm{ is a main vertex, } \\\, \frac{y(e)}{\mu} \quad \textrm{if } i
\textrm{ is an additional vertex of the edge } e^r. \end{cases}\]
If $ u $ is a main vertex, we have
	\begin{eqnarray*}
		(\Q(\h^r)\x)_u &=& \sum_{e^r \in E(\h^r)_{[u]}}x(e^r)\\
		&=& \sum_{e \in E(\h)_{[u]}}\left( y(e)+(r-k)\frac{y(e)}{\mu}\right)\\
		&=& \sum_{e \in E(\h)_{[u]}}y(e)+\left( \frac{r-k}{\mu}\right)\!\!\! \sum_{e \in E(\h)_{[u]}}y(e)\\
		&=& (\mu+r-k)y_u = (\mu+r-k)x_u.
	\end{eqnarray*}
	
	Now, if $ u $ is an additional vertex, we have
	$$(\Q(\h^r)\x)_u = x(e^r) =  y(e)+(r-k)\frac{y(e)}{\mu} =
(\mu+r-k)\frac{y(e)}{\mu} = (\mu+r-k)x_u.$$
	Therefore, the result follows.
\end{proof}

\begin{Lem}
Let $\h$ be a $k$-graph and $s\geq 1$ an integer. If $(\lambda,\x)$ is a
signless Laplacian eigenpair of $\h_s$, with $\lambda > 0$, then for each
edge $e_s \in E(\h_s)$, we have $x(e_s) = sx(e)$.
\end{Lem}
\begin{proof}
By Lemma \ref{teoadc}, we have $x(\varsigma_u) = x_u+x_{u^2}+\cdots+x_{u^s} =
sx_u$. Hence,
	$$x(e_s) = x(\varsigma_{u_1})+\cdots+x(\varsigma_{u_k}) =
s(x_{u_1}+\cdots+x_{u_k}) = sx(e).$$
\end{proof}

\begin{Pro}
	Let $\h$ be a $k$-graph and $s\geq 1$ an integer. If $\mu\neq0$ is a signless Laplacian eigenvalue of $\h$, then $\lambda= s\mu$ is an eigenvalue of $\Q(\h_s)$.
\end{Pro}
\begin{proof}
	Suppose $\y$ is an eigenvector of $\Q(\h)$ associated with $\mu$. Define a vector $\x$ of dimension $|V(\h_s)|$, by $x_u = y_v$, if $u \in \varsigma_v$.
	Thus,
	\[(\Q(\h_s)\x)_u = \sum_{e_s\in E(\h_s)_{[u]}} x(e_s) =
\sum_{e \in E(\h)_{[u]}}sx(e) = s\mu x_u.\]
\end{proof}

\begin{Teo}\label{teo:q-qr}
	Let $\h$ be $k$-graph, $s\geq 1$ and $r \geq ks$ be two integers.
	$\lambda>r-ks$ is an eigenvalue of $\Q(\h^r_s)$ if and only if there
	is a signless Laplacian eigenvalue $\mu>0$ of $\h$ such that
$\lambda= s(\mu-k)+r$.	
\end{Teo}
\begin{proof}
If $\mu$ is a signless Laplacian eigenvalue of $\h$, then $s\mu$ is an
eigenvalue of $\Q(\h_s)$. So, $\lambda=s\mu+r-ks = s(\mu-k)+k$ is a signless
Laplacian eigenvalue of $(\h_s)^r=\h^r_s$.

Now, let $\x$ be an eigenvector associated with $\lambda$ in $\Q(\h^r_s)$. Thus,
	\begin{eqnarray*}
		\lambda x_u &=& \sum_{e^r_s\in E(\h^r_s)_{[u]}}x(e^r_s)\\
		&=& \sum_{e^r_s\in E(\h^r_s)_{[u]}}(r-ks)x_{i_e^1}+x(e_s)\\
		&=& \sum_{e_s\in E(\h_s)_{[u]}}(r-ks)\frac{x(e_s)}{\lambda-r+ks}+x(e_s)\\
		&=& \left( \frac{\lambda s}{\lambda+ks-r}\right)\sum_{e \in E(\h)_{[u]}}x(e).
	\end{eqnarray*}
	Therefore,
	$$\sum_{e \in E(\h)_{[u]}}x(e) = \frac{\lambda+ks-r}{s} x_u.$$
	That is, $\h$ has a signless Laplacian eigenvalue $\mu$, such that
	$$ \mu = \frac{\lambda+ks-r}{s}\quad \Rightarrow \quad \lambda= s(\mu-k)+r.$$
\end{proof}

We notice that Theorem \ref{teo:q-qr} characterizes all signless Laplacian
eigenvalues greater than $r-ks$ of a power hypergraph $\h^r_s$. Now we will
study the other eigenvalues.

\begin{Pro}\label{pro:multzerohs}
Let $\h$ be a $k$-graph. If $s \geq 1$ is an integer, then the multiplicity
of $\lambda = 0$ as eigenvalue of $\Q(\h_s)$ is $s|V|-t$. Where $t$ is the
rank of the matrix $\Q(\h)$.
\end{Pro}
\begin{proof}
If $\z$ is an eigenvector of $\lambda = 0$ in $\Q(\h)$, define a new vector
$\x$ of dimension $|V(\h_s)|$, by $x_v = z_u$ if $v \in \varsigma_u$. Notice that,
	$$(\Q(\h_s)\x)_v = \sum_{e_s \in E(\h_s)_{[v]}}x(e_s) =
s\!\!\!\!\!\!\sum_{e \in E(\h)_{[v]}}z(e) = 0.$$
Hence, for each eigenvector of $\lambda = 0$ in $\Q(\h)$, we build one
for $\h_s$,  i.e., we construct a family of $|V|-t$ linearly independent
eigenvectors.
	
Now, for each $v \in V(\h)$, suppose $\varsigma_v = \{v,v_2\ldots,v_s\}$ and
$2 \leq j \leq s$. We can construct the following family of $s-1$ linearly
independent vectors.
	$$\x^j = \begin{cases}(x^j)_{v}\; =\;\;\, 1,\\
	(x^j)_{v_j} = -1,\\
	(x^j)_{u}\; =\;\;\, 0, \textrm{ for } u \in V(\h_s)-\{v_1,v_j\}.\end{cases}$$
	
Notice that these vectors are eigenvectors of $\lambda = 0$ in $\Q(\h_s)$.
Repeating this construction for the other main vertices of $\h_s$, we obtain
$(s-1)|V|$ linearly independent eigenvectors. Observe that these vectors are
linearly independent from those constructed from the zero eigenvectors of the
base hypergraph $\h$. To see this. we observe that the former vectors have
constant sign in each $ \varsigma_u $, while these new vectors have more than
one sign in these sets. Therefore we have $s|V|-t$ linearly independent
eigenvectors of $\lambda=0$.
\end{proof}

\begin{Pro}
Let $\h$ be a $k$-graph. If $s \geq 1$ and $r > ks$ are two integers, then
the multiplicity of $\lambda = 0$ as eigenvalue of $\Q(\h^r_s)$ is at least
$(r-ks-1)|E|+s|V|$.
\end{Pro}
\begin{proof}
	Let $e \in E(\h)$ be an edge, suppose $\varsigma_e = \{u_1,\ldots,u_{r-ks}\}$ and $2 \leq j \leq r-ks$. Similarly to the proof of Proposition \ref{pro:multzerohs}, we can construct the following family of $r-ks-1$ linearly independent vectors.
	$$\y^j = \begin{cases}(y^j)_{u_1} =\;\;\, 1,\\
	(y^j)_{u_j} = -1,\\
	(y^j)_{u}\; =\;\;\, 0, \textrm{ for } u \in
V(\h^r_s)-\{u_1,u_j\}.\end{cases}$$ Repeating this construction for the other
edges of $\h$, we obtain $(r-ks-1)|E|$ linearly independent eigenvectors,
associated with $\lambda = 0$.
	
Now, let $w \in V(\h_s)$, and consider $e_1,\ldots,e_p$, all edges of $\h^r_s$
that contain the vertex $w$. For each of this, take $w_i\in e_i$ an additional
vertex. So we can build the vector
	$$\z = \begin{cases}z_{w} =\;\;\, 1,\\
	z_{w_i} = -1, \textrm{ for } 1 \leq i \leq p,\\
	z_{u}\; =\;\;\, 0, \textrm{ for } u \in V(\h^r_s)-\{w,w_1,\ldots,w_p\}.\end{cases}$$
Repeating this construct for the other vertices of $\h_s$, we obtain $s|V|$ eigenvectors associated with $\lambda = 0$, linearly independent to each other and with the others previously created. Totalizing $(r-ks-1)|E| + s|V|$ eigenvectors.
\end{proof}

\begin{Teo}
	Let $\h$ be a $k$-graph. If $s \geq 1$ and $r > ks$ are integers, then the multiplicity of $\lambda = r-ks$ as eigenvalue of $\Q(\h^r_s)$ is $|E|-t$. Where $t$ is the rank of the signless Laplacian matrix $\Q(\h)$.
\end{Teo}
\begin{proof}
	Firstly, note that $-k$ is an eigenvalue of multiplicity $|E|-t$ from $\A_\lin$. Let $\z = (z_{e_1},\ldots,z_{e_m})$ be a eigenvector of $-k$ in $\A_\lin$. Note that $$\sum_{v \in V}\left(\sum_{e \in E_{[v]}}z_e \right)^2 = \z^T\B^T\B\z = 0, \quad \Rightarrow \quad \sum_{e \in E_{[v]}}z_e = 0,\;\; \forall v \in V.$$
	
	Now, define a vector $\x$ of dimension $|V(\h^r_s)|$, by
	$$\x = \begin{cases}x_v = \,z_e, \textrm{ if } v \textrm{ is an adictional vertice of } e,\\
	x_v = \;0,\,\, \textrm{ if } v \textrm{ is not an adictional vertice.}\end{cases} $$
	If $u$ is an adictional vertice, then
	$$(\Q(\h^r_s)\x)_u = x(e) = (r-ks)(z_e) = (r-ks)x_u.$$
	If $u$ is a main or copy vertice, then
	$$(\Q(\h^r_s)\x)_u = \sum_{e \in E_{[u]}}x(e) = \sum_{e \in E_{[v]}}z_e = 0 = (r-ks)x_u.$$
	Therefore, the result follows.	
\end{proof}

\begin{Obs}
If $\h$ is a $k$-graph with $n$ vertices, $m$ edges having
signless Laplacian eigenvalues $\lambda_1\geq\lambda_2\geq\cdots\geq\lambda_t
>\lambda_{t+1}=\cdots=\lambda_n=0$, then
the eigenvalues of $\Q(\h^r_s)$ are $s(\lambda_1-k)+r, \ldots,
s(\lambda_t-k)+r$, $r-ks$ with multiplicity $m-t$ and  $0$ with
multiplicity $(r-ks-1)m+sn$.
\end{Obs}

\section*{Acknowledgments} This work is part of the
doctoral studies of K. Cardoso under the supervision of V.~Trevisan. K.
Cardoso is grateful for the support given by Instituto Federal do Rio Grande
do Sul (IFRS), Campus Feliz. V. Trevisan acknowledges partial support of CNPq grants
409746/2016-9 and 303334/2016-9, CAPES (Proj. MATHAMSUD 18-MATH-01) and
FAPERGS (Proj.\ PqG 17/2551-0001).

\bibliographystyle{acm}
\bibliography{Bibliografia}

\end{document}